\documentclass[submission]{FPSAC2018}
\usepackage{verbatim}
\usepackage{float}

\usepackage{times}
\usepackage[T1]{fontenc}
\usepackage{mathrsfs}
\usepackage{epsfig}
\usepackage{color}
\usepackage{array}
\usepackage{enumitem}

\newcolumntype{L}{>{$}l<{$}}
\newcolumntype{R}{>{$}r<{$}}

\newtheorem{thm}{Theorem}[section]

\newtheorem{cor}[thm]{Corollary}

\theoremstyle{definition}

\newtheorem{example}[thm]{Example}

\numberwithin{equation}{section}

\def\ZZ{\mathbb{Z}}
\def\QQ{\mathbb{Q}}

\def\FF{\mathbb{F}}

\def\CC{\mathbb{C}}
\def\RR{\mathbb{R}}

\def\Conf{\mathrm{Conf}}

\newcommand\PConf{\mathrm{PConf}}

\def\sgn{\mathrm{sgn}}
\def\multi#1#2{\ensuremath{\left(\kern-.3em\left(\genfrac{}{}{0pt}{}{#1}{#2}\right)\kern-.3em\right)}}

\newcommand{\poly}{\mathrm{Poly}}
\newcommand{\sfr}{\mathrm{sf}}

\newcommand{\Sp}{\mathcal{S}}
\newcommand{\tr}{\mathbf{1}}
\newcommand{\Sgn}{\mathbf{Sgn}}
\newcommand{\Std}{\mathbf{Std}}

\received{\today}

\usepackage[backend=bibtex]{biblatex}
\abstract{Factorization statistics are functions defined on the set $\poly_d(\FF_q)$ of all monic degree $d$ polynomials with coefficients in $\FF_q$ which only depend on the degrees of the irreducible factors of a polynomial. We show that the expected values of factorization statistics are determined by the representation theoretic structure of the cohomology of point configurations in $\RR^3$. This \emph{twisted Grothendieck-Lefschetz formula for $\poly_d$} is analogous to a result of Church, Ellenberg, and Farb for \emph{squarefree} polynomials. Our proof uses formal power series methods which also lead to a new proof of the Church, Ellenberg, and Farb result circumventing algebraic geometry.}

\begin{document}

\title[Factorization Statistics]{Factorization Statistics and the\\ Twisted Grothendieck-Lefschetz formula}

\author{Trevor Hyde\thanks{\href{mailto:tghyde@umich.edu}{tghyde@umich.edu}}}

\address{Dept. of Mathematics\\
University of Michigan \\
Ann Arbor, MI 48109-1043\\
}

\maketitle

%%%Reviewer comment: They want some indication of proofs. Look for things that may be cut and find a way to include something.

\section{Introduction}
What is the probability that a random integer $m$ in the interval $[1,n]$ is prime? The Prime Number Theorem tells us that
\[
    \mathrm{Prob}(m \in [1,n]\text{ is prime}) \approx \frac{1}{\log(n)}
\]
for sufficiently large $n$. Following a classic analogy between $\ZZ$ and $\FF_q[x]$, we ask: what is the  probability that a random monic degree $d$ polynomial $f(x) \in \FF_q[x]$ is irreducible? One can show that
\begin{equation}
\label{eqn 1}
    \mathrm{Prob}(f(x)\text{ monic degree $d$ is irreducible}) \approx \frac{1}{d}
\end{equation}
for large values of $q$. Note that the number of monic degree $d$ polynomials in $\FF_q[x]$ is $q^d$, hence $\frac{1}{d} = \frac{1}{\log_q (q^d)}$ parallels the result for $\ZZ$. This is the beginning of a motivating theme: analogous arithmetic statistical questions for $\ZZ$ and $\FF_q[x]$ have analogous answers. 

Often we may determine the exact values of statistics on the $\FF_q[x]$ side of the analogy which seem out of reach for $\ZZ$.  For example, the number of irreducible monic degree $d$ polynomials in $\FF_q[x]$ is given by $d$th \emph{necklace polynomial},
\[
    \#\{f(x) \text{ irreducible monic degree $d$}\} = \frac{1}{d}\sum_{e\mid d}\mu(d/e)q^e,
\]
where $\mu$ is the M\"{o}bius function. Therefore,
\begin{equation}
\label{eqn irred prob} 
    \mathrm{Prob}(f(x)\text{ monic degree $d$ is irreducible}) = \frac{1}{d}\sum_{e\mid d}\frac{\mu(d/e)}{q^{d-e}}.
\end{equation}

From an analytic point of view there is an impulse to focus on the leading terms for a probability like \eqref{eqn 1}. However, on closer inspection of the precise formula \eqref{eqn irred prob}, we see that each term has a structural interpretation. Consider the case of \eqref{eqn irred prob} when $d = 6$,
\[
    \mathrm{Prob}(f(x)\text{ monic degree $6$ is irreducible}) = \frac{1}{6}\bigg(1 - \frac{1}{q^3} - \frac{1}{q^4} + \frac{1}{q^5}\bigg).
\]
The four terms in this expression correspond to the intermediate fields of the degree 6 extension $\FF_{q^6}/\FF_q$ and the coefficients encode how these fields fit together.

This brings us to our main thesis: the exact expressions for arithmetic statistical questions in $\FF_q[x]$ reflect hidden structure which is not apparent from approximations. In other words, \emph{there are no error terms}, each term has an interpretation and together they tell a complete story. Our main result (Theorem \ref{thm twist intro}) supports this claim for the expected values of functions on $\FF_q[x]$.

A \emph{factorization statistic} $P$ is a function defined on the set $\poly_d(\FF_q)$ of monic degree $d$ polynomials in $\FF_q[x]$ such that $P(f)$ depends only on the partition of $d$ given by the degrees of the irreducible factors of $f$.
$P$ may also be viewed as a function defined on partitions of $d$, or as a class function of the symmetric group $S_d$. Let $\psi_d^k$ be the character of the $S_d$-representation $H^{2k}(\PConf_d(\RR^3),\QQ)$ where $\PConf_d(\RR^3)$ is the ordered configuration space of $d$ distinct points in $\RR^3$ (see Section \ref{section twist}.) 

\begin{thm}[Twisted Grothendieck-Lefschetz for $\poly_d$]
\label{thm twist intro}
If $P$ is a factorization statistic, then the expected value $E_d(P)$ of $P$ on $\poly_d(\FF_q)$ is given by
\[
    E_d(P) := \frac{1}{q^d}\sum_{f\in \poly_d(\FF_q)}P(f) = \sum_{k=0}^{d-1}\frac{\langle P, \psi_d^k\rangle}{q^k},
\]
where $\langle P, \psi_d^k\rangle = \frac{1}{d!}\sum_{\sigma\in S_d} P(\sigma)\psi_d^k(\sigma)$ is the standard inner product of class functions of the symmetric group $S_d$.
\end{thm}

The twisted Grothendieck-Lefschetz formula provides a bridge between representation theory and topology on the one hand and the combinatorics of finite fields on the other. We explore this interplay through examples in Section \ref{section example}.

\section{Factorization statistics}
\label{section fact}
The \emph{factorization type} of a polynomial $f(x) \in \FF_q[x]$ is the partition of $\deg(f)$ given by the degrees of the irreducible factors of $f(x)$. Let $\poly_d(\FF_q)$ denote the set of monic degree $d$ polynomials in $\FF_q[x]$. A \emph{factorization statistic} $P$ is a function defined on polynomials $f(x) \in \poly_d(\FF_q)$ which only depends on the factorization type of $f(x)$.

\begin{example}\hspace{.1in}

\label{ex 1}
\begin{enumerate}[leftmargin=*]
    \item Consider the polynomials $g(x), h(x) \in \poly_5(\FF_3)$ with irreducible factorizations
    \[
        g(x) = x^2(x+1)(x^2 + 1) \hspace{3em} h(x) = (x+1)(x - 1)(x^3 - x + 1).
    \]
    The factorization type of $g(x)$ is the partition $[2,1,1,1]$ and the factorization type of $h(x)$ is $[3,1,1]$. Note that the factorization type does not record the multiplicity of a specific factor so that $x^2$ and $x(x + 1)$ both have the same factorization type $[1,1]$.
    
    \item Let $R(f)$ be the number of $\FF_q$-roots of $f(x)\in \poly_d(\FF_q)$. Then $R(f)$ depends only on the number of linear factors of $f(x)$, hence is a factorization statistic. Referring to the two polynomials above we have $R(g) = 3$ and $R(h) = 2$.
    
    \item Say a polynomial $f(x)$ has \emph{even type} if the factorization type of $f(x)$ is an even partition. More specifically, say $\lambda = (1^{m_1} 2^{m_2} 3^{m_3}\cdots )$ is the factorization type of $f(x)$ and define $\sgn(\lambda)$ by
    \[
        \sgn(\lambda) = \prod_{j\geq 1} (-1)^{m_j(j - 1)},
    \]
    then $f(x)$ has even type if $\sgn(\lambda) = 1$. The function $ET$ defined by
    \[
        ET(f) = \begin{cases} 1 & \text{$f(x)$ has even type}\\ 0 & \text{otherwise,}\end{cases}
    \]
    is a factorization statistic. Continuing our examples, $ET(g) = 0$ and $ET(h) = 1$.
    
    \item Define the \emph{quadratic excess} $Q(f)$ of a polynomial $f$ to be
    \[
        Q(f) = \#\{\text{reducible quadratic factors of $f(x)$}\} - \#\{\text{irreducible quadratic factors of $f(x)$}\}.
    \]
    Then $Q(f)$ depends only on the number of linear and irreducible quadratic factors of $f(x)$, hence is a factorization statistic. Since $g(x)$ has 3 linear factors and 1 irreducible quadratic factor, we have
    \[
        Q(g) = \binom{3}{2} - 1 = 2.
    \]
    The polynomial $h(x)$ has 2 linear factors and 0 irreducible quadratic factors, hence
    \[
        Q(h) = \binom{2}{2} - 0 = 1.
    \]
\end{enumerate}

\end{example}

Let $E_d(P)$ denote the expected value of a factorization statistic $P$ on the set $\poly_d(\FF_q)$ of all monic degree $d$ polynomials. More precisely,
\[
    E_d(P) := \frac{1}{q^d}\sum_{f\in \poly_d(\FF_q)}P(f).
\]
By counting the number of polynomials with a given factorization type (e.g. using unique factorization and necklace polynomials) we can explicitly compute $E_d(P)$ for any particular $P$ and $d$ as a function of $q$. For example, here are some computations of $E_d(Q)$ where $Q$ is the quadratic excess statistic defined in Example \ref{ex 1} (4).
\begin{center}
\begin{tabular}{c|l}
    $d$ & $E_d(Q)$\\
\hline
    $3$ & $\tfrac{2}{q} + \tfrac{1}{q^2}$ \\
    $4$ & $\tfrac{2}{q} + \tfrac{2}{q^2} + \tfrac{2}{q^3}$\\
    $5$ & $\tfrac{2}{q} + \tfrac{2}{q^2} + \tfrac{4}{q^3} + \tfrac{2}{q^4}$\\
    $6$ & $\tfrac{2}{q} + \tfrac{2}{q^2} + \tfrac{4}{q^3} + \tfrac{4}{q^4} + \tfrac{3}{q^5}$\\
    $10$ & $\tfrac{2}{q} + \tfrac{2}{q^2} + \tfrac{4}{q^3} + \tfrac{4}{q^4} + \tfrac{6}{q^5} + \tfrac{6}{q^6} + \tfrac{8}{q^7} + \tfrac{8}{q^8} + \tfrac{5}{q^9}$\\
\end{tabular}
\end{center}

\noindent There are some remarkable features of these expected values: $E_d(Q)$ is a polynomial in $\frac{1}{q}$ of degree $d - 1$ with \emph{positive integer coefficients}---one should expect the coefficients to be rational numbers, but both the positivity and integrality are not a priori evident. Evaluating the polynomial $E_d(Q)$ at $q = 1$ gives the binomial coefficient $\binom{d}{2}$. The coefficients of $E_d(Q)$ appear to stabilize as $d$ increases with a clear pattern emerging already for $d = 10$, suggesting that the expected values $E_d(Q)$ converge coefficientwise as $d\rightarrow\infty$. We return to this example in Section \ref{sec quad} to explain these observations in the light of our results.

\section{Twisted Grothendieck-Lefschetz formulas}
\label{section twist}
We briefly detour from our discussion of factorization statistics and finite fields to review some topology. If $X$ is a topological space, then $\PConf_d(X)$ is
\[
    \PConf_d(X) := \{(x_1, x_2, \ldots, x_d) \in X^d : x_i \neq x_j\},
\]
the \emph{ordered configuration space of $d$ points on $X$}. The symmetric group $S_d$ acts on $\PConf_d(X)$ by permuting coordinates; this action is free given that all coordinates are distinct. Let $\Conf_d(X)$ be the quotient of $\PConf_d(X)$ by this action. $\Conf_d(X)$ is the space of \emph{unordered configurations of $d$ points on $X$}. Note that for each $k\geq 0$ the singular cohomology $H^k(\PConf_d(X),\QQ)$ is, by functoriality, a finite dimensional $S_d$-representation.

Our main result establishes a surprising connection between the expected values of factorization statistics on $\poly_d(\FF_q)$ and the sequence of $S_d$-representations $H^{2k}(\PConf_d(\RR^3),\QQ)$. The cohomology $H^\ast(\PConf_d(\RR^3),\QQ)$ is supported in even degrees, vanishing beyond degree $2(d-1)$.

\begin{thm}[Twisted Grothendieck-Lefschetz formula for $\poly_d$]
\label{twisted gl}
Let $P$ be a factorization statistic and let $\psi_d^k$ be the character of the $S_d$-representation $H^{2k}(\PConf_d(\RR^3),\QQ)$. Then the expected value $E_d(P)$ of $P$ on the set $\poly_d(\FF_q)$ of polynomials is given by
\[
	E_d(P) := \frac{1}{q^d}\sum_{f\in \poly_d(\FF_q)}P(f) = \sum_{k=0}^{d-1} \frac{\langle P, \psi_d^k\rangle}{q^k},
\]
where $\langle P, \psi_d^k\rangle = \frac{1}{d!}\sum_{\sigma\in S_d} P(\sigma)\psi_d^k(\sigma)$ is the standard inner product of $S_d$-class functions.
\end{thm}

Theorem \ref{twisted gl} shows that the coefficients of the expected value $E_d(P)$ are determined by the representation theoretic structure of $H^{\ast}(\PConf_d(\RR^3),\QQ)$ for any factorization statistic $P$. Note that the factorization statistic $P$ plays different roles on each side of this equation: on the left it acts as a function on the set $\poly_d(\FF_q)$ of polynomials; on the right it acts as a class function of the symmetric group $S_d$.

The sequence of representations $H^{2k}(\PConf_d(\RR^3),\QQ)$ has another interpretation as the \emph{higher Lie representations} $\mathrm{Lie}_k$ \cite[Sec. 2.6]{HershReiner}. We express Theorem \ref{twisted gl} in terms of the cohomology of point configurations in $\RR^3$ to parallel the following result of Church, Ellenberg, and Farb.

\begin{thm}[{\cite[Prop. 4.1]{CEF}}]
\label{twisted gl sf}
Let $P$ be a factorization statistic, and let $\phi_d^k$ be the character of the $S_d$-representation $H^k(\PConf_d(\RR^2), \QQ)$. If $\poly_d^\sfr(\FF_q)$ is the set of squarefree polynomials of degree $d$ in $\FF_q[x]$, then
\[
    \frac{1}{q^d}\sum_{f\in \poly_d^\sfr(\FF_q)}P(f) = \sum_{k=0}^{d-1}\frac{(-1)^k \langle P, \phi_d^k \rangle}{q^k}.
\]
\end{thm}

The proof of Theorem \ref{twisted gl sf} in \cite{CEF} uses algebraic geometry: viewing $\PConf_d$ as a scheme defined over $\ZZ$, the Grothendieck-Lefschetz trace formula for \'{e}tale cohomology with ``twisted coefficients'' expresses the weighted point counts on $\Conf_d(\FF_q)$ in terms of the trace of Frobenius. This combined with a purity result and a comparison theorem between \'{e}tale and singular cohomology yields Theorem \ref{twisted gl sf}.

To see the connection between squarefree polynomials and point configurations in $\RR^2$, we view the plane as $\CC$ and note there is a natural correspondence between squarefree polynomials over $\CC$ of degree $d$ and unordered configurations of $d$ distinct points in $\CC$: 
\[
    (x - \alpha_1)(x - \alpha_2)\cdots(x - \alpha_d)\hspace{1em} \longleftrightarrow \hspace{1em} \{\alpha_1, \alpha_2, \ldots, \alpha_d\}
\]
This correspondence extends to an isomorphism of schemes $\poly_d^\sfr \cong \Conf_d$.

The geometric perspective behind their proof appears to break down in the case of Theorem \ref{twisted gl}, as there is no clear connection between general polynomials of degree $d$ and point configurations in $\RR^3$. Instead we prove Theorem \ref{twisted gl} using a generating function argument. Our approach also leads to a new proof of Theorem \ref{twisted gl sf}, circumventing the methods of algebraic geometry entirely.

We outline the strategy for Theorem \ref{twisted gl}. A complete proof appears in \cite{Hyde1}. Consider the \emph{splitting measure} $\nu$ defined on partitions $\lambda \vdash d$ by
\[
    \nu(\lambda) = \mathrm{Prob}(f\in \poly_d(\FF_q) \text{ has factorization type $\lambda$}).
\]
One may show that $\nu(\lambda)$ is a polynomial in $1/q$ with rational coefficients for any partition $\lambda$. The connection between $\poly_d(\FF_q)$ and $H^{\ast}(\PConf_d(\RR^3),\QQ)$ is made through the following result.

\begin{thm}[{\cite[Thm. 1.4]{Hyde1}}]
\label{thm split}
Let $\psi_d^k$ be the character of $H^{2k}(\PConf_d(\RR^3),\QQ)$. If $\lambda \vdash d$ is a partition, let $z_\lambda = \prod_{j\geq 1} j^{m_j} m_j!$ when $\lambda = (1^{m_1}2^{m_2}\cdots)$. Then
\[
    \nu(\lambda) = \frac{1}{z_\lambda}\sum_{k=0}^{d-1}\frac{\psi_d^k(\lambda)}{q^k}.
\]
\end{thm}
Theorem \ref{thm split} is deduced with the help of a beautiful product formula for the cycle index series of the family $H^{\ast}(\PConf_d(\RR^3),\QQ)$ of representations which may be found in Hersh and Reiner \cite[Thm. 2.7]{HershReiner}. Once we have this result, Theorem \ref{twisted gl} follows by a change in the order of summation.

\begin{proof}[Proof of Theorem \ref{twisted gl}]
Since factorization statistics depend only on the factorization type of a polynomial, we may rewrite the expected value in terms of the splitting measure,
\[
    E_d(P) = \frac{1}{q^d}\sum_{f\in \poly_d(\FF_q)} P(f) = \sum_{\lambda \vdash d} P(\lambda)\nu(\lambda).
\]
Then Theorem \ref{thm split} implies,
\[
    \sum_{\lambda \vdash d} P(\lambda)\nu(\lambda)
    = \sum_{\lambda\vdash d}\frac{1}{z_\lambda}\sum_{k=0}^{d-1} \frac{P(\lambda)\psi_d^k(\lambda)}{q^k}
    = \sum_{k=0}^{d-1} \frac{1}{q^k}\left(\sum_{\lambda \vdash d} \frac{P(\lambda)\psi_d^k(\lambda)}{z_\lambda}\right)
    = \sum_{k=0}^{d-1} \frac{\langle P, \psi_d^k\rangle}{q^k}.
\]
\end{proof}

Church, Ellenberg, and Farb combine Theorem \ref{twisted gl sf} with the \emph{representation stability} of\\ $H^k(\PConf_d(\RR^2),\QQ)$ to deduce the asymptotic stability of squarefree factorization statistics. A sequence $V_d$ of $S_d$-representations is called representation stable when the decomposition of $V_d$ into irreducibles stabilizes as $d\rightarrow\infty$. We refer the reader to \cite{ChurchFarb} for a precise description of representation stability. For us the important fact is the following: Let $x_j$ be the function defined on partitions where $x_j(\lambda)$ is the number of parts of $\lambda$ os size $j$. Suppose $P$ is an element in $\QQ[x_1, x_2,\ldots]$, then $P$ defines a function on partitions called a \emph{character polynomial}. If $V_d$ is a representation stable sequence with character $\chi_d$ and $P$ is a character polynomial, then the following limit exists
\[
    \langle P, \chi\rangle = \lim_{d\rightarrow\infty}\langle P, \chi_d\rangle.
\]
Furthermore, the sequence $\langle P, \chi_d\rangle$ is eventually constant (see \cite[Sec. 3.4]{CEF}.) It follows from a general result of Church \cite[Thm. 1]{Church} that for each $k$ the sequence $H^{2k}(\PConf_d(\RR^3),\QQ)$ is representation stable. Thus we have the following corollary of Theorem \ref{twisted gl}.

\begin{cor}[Asymptotic stability of expected values]
\label{cor converge}
If $P$ is a factorization statistic given by a character polynomial, then the following limit converges coefficientwise in the ring of formal power series in $1/q$:
\[
    \lim_{d\rightarrow\infty}E_d(P) = \sum_{k=0}^\infty \frac{\langle P, \psi^k \rangle}{q^k},
\]
where $\langle P, \psi^k \rangle := \lim_{d\rightarrow\infty} \langle P, \psi_d^k \rangle$.
\end{cor}

\section{Examples}
\label{section example}

In this section we explore the interplay provided by Theorem \ref{twisted gl} between finite field combinatorics, representation theory, and topology through examples.

\subsection{Quadratic excess}
\label{sec quad}
Recall the quadratic excess factorization statistic $Q$ from Section \ref{section fact}: $Q(f)$ is defined as the difference between the number of reducible versus irreducible quadratic factors of $f$. Rephrasing this in terms of partitions we see that $Q$ is given by the character polynomial
\[
    Q(\lambda) = \binom{x_1(\lambda)}{2} - \binom{x_2(\lambda)}{1}.
\]
Let $\QQ[d]$ be the permutation representation of the symmetric group with basis $\{e_1, e_2, \ldots , e_d\}$ and consider the linear representation given by the second exterior power $\bigwedge^2 \QQ[d]$. This representation has dimension $\binom{d}{2}$ with basis given by $\{e_i \wedge e_j : i < j\}$. If $\sigma \in S_d$ is a permutation, then the trace of $\sigma$ on $\bigwedge^2 \QQ[d]$ is
\begin{align*}
    \mathrm{Trace}(\sigma) &= \#\{\{i,j\} : \sigma \text{ fixes $i$ and $j$}\} - \#\{\{i,j\} : \sigma \text{ transposes $i$ and $j$}\}\\
    &= \binom{x_1(\sigma)}{2} - \binom{x_2(\sigma)}{1}\\
    &= Q(\sigma).
\end{align*}
Thus $Q$, viewed as a class function of $S_d$, is the character of $\bigwedge^2\QQ[d]$. It follows that $\langle Q, \psi_d^k \rangle$ is a non-negative integer for all $d, k \geq 0$. This together with Theorem \ref{twisted gl} explains the non-negative integral coefficients of $E_d(Q)$. That the degree of $E_d(Q)$ is $d - 1$ reflects that $2(d - 1)$ is the largest non-vanishing degree of cohomology for $\PConf_d(\RR^3)$. The coefficientwise convergence of $E_d(Q)$ follows from Corollary \ref{cor converge}.

The coefficientwise convergence of $E_d(P)$ holds in much greater generality for functions $P$ defined on $\Conf_d(V)$ for $V$ an affine or projective variety defined over $\FF_q$ which only depend on the cycle structure of Frobenius, even when there is no apparent representation stability present; we do not pursue this further here but refer the reader to \cite[Cor. 10]{Chen}. One benefit of the combinatorial approach is that we can explicitly compute the limits of $E_d(P)$ as a rational function of $q$. For example,
\begin{align*}
    \lim_{d\rightarrow\infty}E_d(Q) &= \frac{1}{2}\bigg(1 + \frac{1}{q}\bigg)\bigg(\frac{1}{1 - \frac{1}{q}}\bigg)^2 - \frac{1}{2}\bigg(1 - \frac{1}{q}\bigg)\bigg(\frac{1}{1 - \frac{1}{q^2}}\bigg)\\
    &= \frac{2}{q} + \frac{2}{q^2} + \frac{4}{q^3} + \frac{4}{q^4} + \frac{6}{q^5} + \frac{6}{q^6} + \frac{8}{q^7} + \frac{8}{q^8} + \frac{10}{q^9} + \ldots
\end{align*}

\subsection{Constraint on total cohomology}
The next result gives a constraint on the total cohomology of $\PConf_d(\RR^3)$.
\begin{thm}
\label{thm constraint}
For each $d\geq 0$ there is an isomorphism of $S_d$-representations
\begin{equation}
\label{eqn cohom}
    \bigoplus_{k=0}^{d-1}H^{2k}(\PConf_d(\RR^3),\QQ) \cong \QQ[S_d],
\end{equation}
where $\QQ[S_d]$ is the regular representation of $S_d$.
\end{thm}

\begin{proof}
Let $\rho$ be the character of $\bigoplus_{k=0}^{d-1}H^{2k}(\PConf_d(\RR^3),\QQ)$. Then
\[
    \rho = \sum_{k=0}^{d-1} \psi_d^k,
\]
where $\psi_d^k$ is the character of $H^{2k}(\PConf_d(\RR^3),\QQ)$. By Theorem \ref{thm split} we have
\[
    \nu(\lambda) = \frac{1}{z_\lambda}\sum_{k=0}^{d-1} \frac{\psi_d^k(\lambda)}{q^k},
\]
where $\nu$ is the splitting measure. Let $\nu_1$ denote the splitting measure evaluated at $q = 1$. Then $\nu_1(\lambda) = \frac{\rho(\lambda)}{z_\lambda}$. On the other hand we can compute $\nu_1(\lambda)$ directly. The number of irreducible polynomials in $\poly_d(\FF_q)$ is given by
\[
    M_d(q) = \frac{1}{d}\sum_{e\mid d}\mu(e)q^{d/e}.
\]
Hence by unique factorization in $\FF_q[x]$,
\[
    \nu(\lambda) = \frac{1}{q^d}\prod_{j\geq 1}\binom{M_j(q) + m_j - 1}{m_j}.
\]
Since $M_j(1) = 0$ for $j > 1$ and $M_1(1) = 1$ it follows that
\[
    \nu_1(\lambda) = \prod_{j\geq 1}\binom{M_j(1) + m_j - 1}{m_j} = \begin{cases} 1 & \lambda = [1^d]\\ 0 & \text{otherwise.} \end{cases}
\]
Since $z_{[1^d]} = d!$, $\nu_1(\lambda) = \frac{\rho(\lambda)}{z_\lambda}$ implies
\[
    \rho(\lambda) = \begin{cases} d! & \lambda = [1^d]\\ 0 & \text{otherwise,} \end{cases}
\]
which is the character of the regular representation.
\end{proof}

The right hand side of \eqref{thm constraint} is well-understood: the irreducible representations of $S_d$ are indexed by partitions $\lambda\vdash d$, each irreducible $\Sp_\lambda$ is a direct summand of $\QQ[S_d]$ with multiplicity $f_\lambda := \dim \Sp_\lambda$. Thus Theorem \ref{thm constraint} tells us that all the irreducible components of $\QQ[S_d]$ are distributed among the various degrees of cohomology on the left hand side of \eqref{eqn cohom}. Theorem \ref{twisted gl} implies that this filtration of the regular representation completely determines and is determined by the expected values of factorization statistics on $\poly_d(\FF_q)$. We use this information to locate some of the irreducible $S_d$-representations in the cohomology of $\PConf_d(\RR^3)$.

\subsubsection{Trivial representation} Let $\tr = \Sp_{[d]}$ be the trivial representation of $S_d$. Recall that the trivial representation $\tr$ is one dimensional with constant character equal to 1. By Theorem \ref{thm constraint} there is precisely one $k$ such that $\tr$ is a summand of $H^{2k}(\PConf_d(\RR^3),\QQ)$. Interpreting the character of $\tr$ as a factorization statistic we have $E_d(1) = 1$ and Theorem \ref{twisted gl} implies
\[
    1 = E_d(1) = \sum_{k=0}^{d-1}\frac{\langle 1, \psi_d^k\rangle}{q^k}.
\]
Comparing coefficients of $1/q^k$ we conclude that $\tr$ is a summand of $H^0(\PConf_d(\RR^3),\QQ)$. On the other hand, $\PConf_d(\RR^3)$ is path connected so the degree 0 cohomology is one dimensional. Thus
\[
    H^0(\PConf_d(\RR^3),\QQ) \cong \tr.
\]
Note that any factorization statistic $P$ is a class function of $S_d$ and the irreducible characters of $S_d$ form a $\QQ$-basis for the vector space of all class functions. Thus there are $a_\lambda\in \QQ$ such that
\[
    P = \sum_{\lambda\vdash d} a_\lambda \chi_\lambda,
\]
where $\chi_\lambda$ is the character of the irreducible representation $\Sp_\lambda$. In particular if $a_1 := a_{[d]}$ is the coefficient of the trivial character in this decomposition, then we have the following corollary.

\begin{cor}
If $P$ is any factorization statistic and $a_1$ is the coefficient of the trivial character in the canonical expression for $P$ as a linear combination of irreducible characters, then
\[
    a_1 = \lim_{q\rightarrow\infty} E_d(P).
\]
Hence $a_1 = 0$ if and only if the expected value of $P$ approaches 0 for large $q$.
\end{cor}

Our table of values for $E_d(Q)$ with $Q$ the quadratic excess show that $\lim_{q\rightarrow\infty}E_d(Q) = 0$ for each $d$, hence the representation $\bigwedge^2\QQ[d]$ has no trivial component.

\subsubsection{Sign representation} The only other one dimensional irreducible representation of $S_d$ is the sign representation $\Sgn := \Sp_{[1^d]}$ whose character we write as $\sgn$. Viewing $\sgn$ as a factorization statistic Theorem \ref{twisted gl} implies
\[
    E_d(\sgn) = \frac{1}{q^k}
\]
for some $k>0$, but which value of $k$ is it?

\begin{thm}
\label{thm sgn}
For each $d\geq 0$,
\[
    E_d(\sgn) = \frac{1}{q^{\lfloor d/2 \rfloor}}.
\]
Hence $H^{2\lfloor d/2 \rfloor}(\PConf_d(\RR^3),\QQ)$ is the unique cohomological degree with a $\Sgn$ summand.
\end{thm}

We prove Theorem \ref{thm sgn} using \emph{liminal reciprocity} which relates factorization statistics in $\poly_d(\FF_q)$ with the limiting values of \emph{squarefree} factorization statistics for $\FF_q[x_1, x_2, \ldots, x_n]$ as the number of variables $n$ tends to infinity. See \cite{Hyde2} for details.

Theorem \ref{thm sgn} has a surprising consequence. Recall that $ET$ is the \emph{even type} factorization statistic defined as $ET(f) = 1$ when the factorization type of $f$ is an even partition and 0 otherwise. Thus the expected value $E_d(ET)$ is the probability of a random polynomial in $\poly_d(\FF_q)$ having even factorization type. One might guess that a polynomial should be just as likely to have an even versus odd factorization type. However, notice that as class functions of $S_d$ we have
\[
    ET = \tfrac{1}{2}(1 + \sgn).
\]
It follows by the linearity of expectation that
\[
    E_d(ET) = \tfrac{1}{2}(E_d(1) + E_d(\sgn)) = \tfrac{1}{2}\big(1 + \tfrac{1}{q^{\lfloor d/2 \rfloor}}\big).
\]
The leading term of this probability is $1/2$ as we expected, but there is a bias toward a polynomial having even factorization type coming from the sign representation and the degree of cohomology in which it appears. For comparison we remark that in the squarefree case the probability of a random polynomial in $\poly_d^\sfr(\FF_q)$ having even factorization type is exactly
\[
    E_d^\sfr(ET) = \tfrac{1}{2},
\]
matching our original guess.

\subsubsection{Standard representation}
Recall the factorization statistic $R$ from Example \ref{ex 1} where $R(f)$ is the number of $\FF_q$-roots of $f(x)$. In \cite{Hyde1} we use generating functions to compute the expected number of roots $E_d(R)$ of a degree $d$ polynomial to be
\begin{equation}
\label{eqn roots}
    E_d(R) = \frac{1 - \frac{1}{q^d}}{1 - \frac{1}{q}} = 1 + \frac{1}{q} + \frac{1}{q^2} + \frac{1}{q^3} + \ldots + \frac{1}{q^{d-1}}.
\end{equation}
Viewed as a class function of $S_d$, $R(\sigma)$ is the number of fixed points of $\sigma$. Hence $R$ is the character of the permutation representation $\QQ[d]$. It is well known that the irreducible decomposition of $\QQ[d]$ is
\[
    \QQ[d] \cong \tr \oplus \Std,
\]
where $\Std := \Sp_{[d-1,1]}$ is the \emph{standard representation} of $S_d$ of dimension $d - 1$. We already determined that $H^0(\PConf_d(\RR^3),\QQ)\cong \tr$, explaining the constant term in \eqref{eqn roots}. Thus Theorem \ref{twisted gl} implies that each $H^{2k}(\PConf_d(\RR^3),\QQ)$ has a single $\Std$ component for $1\leq k \leq d - 1$, accounting for all copies of $\Std$. For comparison we note that
\[
    E_d^\sfr(R) = \frac{1 - \frac{(-1)^{d-1}}{q^{d-1}}}{1 + \frac{1}{q}} = 1 - \frac{1}{q} + \frac{1}{q^2} - \frac{1}{q^3} + \ldots + (-1)^{d-1}\frac{1}{q^{d-2}}. 
\]

\subsubsection{Evaluating at $q = 1$}

Recall that the inner product $\langle \chi, \psi \rangle$ of symmetric group class functions is bilinear. If $P$ is any factorization statistic, then by Theorem \ref{twisted gl} we have the following evaluation of $E_d(P)$ at $q = 1$,
\[
    E_d(P)_{q = 1} = \sum_{k=0}^{d-1} \langle P, \psi_d^k\rangle = \langle P, \sum_{k=0}^{d-1} \psi_d^k \rangle.
\]
Passing to characters in Theorem \ref{thm constraint} gives
\[
    \sum_{k=0}^{d-1} \psi_d^k = \chi_{\mathrm{reg}},
\]
where $\chi_{\mathrm{reg}}$ is the character of the regular representation $\QQ[S_d]$. If $P$ is a character of an $S_d$-representation $V$, then it follows from the general representation theory of finite groups that $\langle P, \chi_{\mathrm{reg}} \rangle = \dim V$. Therefore,
\[
    E_d(P)_{q = 1} = \dim V.
\]
If $Q$ is the quadratic excess factorization statistic, then earlier we showed that $Q$ is the character of the $\binom{d}{2}$-dimensional representation $\bigwedge^2 \QQ[d]$. Hence
\[
    E_d(Q)_{q=1} = \binom{d}{2},
\]
which was observed in the table of values for $E_d(Q)$. We also showed that the root statistic $R$ was the character of the permutation representation $\QQ[d]$, hence
\[
    E_d(R)_{q=1} = d.
\]

\acknowledgements{The author would like to thank Jeff Lagarias for introducing him to splitting measures, Phil Tosteson for valuable references, and Jonathan Gerhard for the observation about specializing expected values at $q = 1$.}

\printbibliography

\end{document}